\documentclass[12pt,A4,reqno]{amsart}
\usepackage{amsfonts}
\usepackage{mathrsfs}

\usepackage{amssymb}
\usepackage{amsmath,amscd}
\usepackage{color}
\usepackage{epsf}
\usepackage{graphicx}
\usepackage{xypic}

\theoremstyle{plain}
\newtheorem{thm}{Theorem}[section]
\newtheorem{lem}[thm]{Lemma}
\newtheorem{prop}[thm]{Proposition}
\newtheorem{cor}[thm]{Corollary}

\theoremstyle{definition}

\newtheorem{rem}[thm]{Remark}

\newcommand{\R}{\mathbb R}
\newcommand{\Z}{\mathbb Z}

\newcommand{\nn}{\vskip 0.2cm}
\newcommand{\n}{\vskip 0.1cm}

\begin{document}

\title [\ ] {Small Cover and Halperin-Carlsson Conjecture - II}

\author{Li Yu}
\address{Department of Mathematics and IMS, Nanjing University, Nanjing, 210093, P.R.China
  \newline
     \textit{and}
 \newline
    \qquad  Department of Mathematics, Osaka City University, Sugimoto,
     Sumiyoshi-Ku, Osaka, 558-8585, Japan}

 \email{yuli@nju.edu.cn}

\thanks{This work is partially supported by
the Scientific Research Foundation for the Returned Overseas Chinese
Scholars, State Education Ministry and by the Japanese Society for
the Promotion of Sciences (JSPS grant no. P10018).}




 \begin{abstract}
   For a small cover $Q^n$ and any principal $(\Z_2)^m$-bundle $M^n$ over $Q^n$,
     it was shown in~\cite{Yu2010-2} that the total sum of $\Z_2$-Betti numbers
      of $M^n$ is at least $2^m$. In this paper,
      we prove that when $M^n$ is connected, the total sum of $\Z_2$-Betti numbers
      of such an $M^n$ exactly equals $2^m$ if and only if
      $M^n$ is homeomorphic to a product of
     spheres, and $Q^n$ in this case must be a generalized real Bott
     manifold (or equivalent, $Q^n$ is a small cover over a product of
     simplices).
 \end{abstract}

\maketitle

  \section{Introduction}

  let $\Z_2$ denote the quotient (additive) group $\Z\slash  2\Z$.
   Based on some basic construction of principal $(\Z_2)^m$-bundles over
   smooth manifolds introduced in~\cite{Yu2010}, the following
   theorem is proved in~\cite{Yu2010-2}. \vskip .2cm

   \begin{thm}[Yu~\cite{Yu2010-2}] \label{thm:Old-Result}
   If $(\Z_2)^m$ acts freely on a manifold $M^n$ whose orbit space is a small cover,
      we must have:
     \begin{equation} \label{Equ:Old-Result}
       \sum^{\infty}_{i=0} \dim_{\Z_2} H^i(M^n,\Z_2) \geq
      2^m.
     \end{equation}
   \end{thm}
    This provides some new evidence to support the Halperin-Carlsson conjecture
    for free $(\Z_2)^m$-actions which claims that
    if $(\Z_2)^m$ can act freely on a finite CW-complex $X$,
    we should have
      \begin{equation} \label{Equ:HCC}
        \sum^{\infty}_{i=0} \dim_{\Z_2} H^i(X,\Z_2) \geq 2^m
      \end{equation}
    The reader is referred to~\cite{Halperin85}---~\cite{Hanke09}
    for more information about the Halperin-Carlsson conjecture.
    In particular, it is interesting see what kind of free $(\Z_2)^m$-actions and $X$
    can make the equality in~\eqref{Equ:HCC} hold. For the sake of
     brevity, we introduce the following notation.
    \[  \mathrm{hrk}(X,\Z_2) :=  \sum^{\infty}_{i=0} \dim_{\Z_2} H^i(X,\Z_2). \]
    We call $\mathrm{hrk}(X,\Z_2)$ the \textit{total $\Z_2$-cohomology rank}
    of $X$.
     In this paper, we will think of a space $X$ with a free $(\Z_2)^m$-action as
      a principal $(\Z_2)^m$-bundle over some base space. The main result of this paper
      is stated as following.
  \vskip .2cm

   \begin{thm} \label{thm:Main}
   For a small cover $Q^n$, there exists some principal
     $(\Z_2)^m$-bundle $M^n$ over $Q^n$ with the total $\Z_2$-cohomology rank
      $\mathrm{hrk}(M^n,\Z_2) =2^m$ if and only if $Q^n$ is a small cover over a product of
     simplices.
   \end{thm} \vskip .2cm
     Recall that an $n$-dimensional \textit{small cover} is a
        closed $n$-manifold with a locally standard $(\Z_2)^n$-action whose orbit space
        can be identified with a simple polytope (see~\cite{DaJan91}).
    \vskip .2cm

   The most obvious examples of $Q^n$ that satisfy the conditions in Theorem~\ref{thm:Main}
    are products of real projective spaces. But in general,
   $Q^n$ could be the total space $B_m$ of an iterated real projective
   space bundle as following:
   \[  B_m \overset{\pi_m}{\longrightarrow} B_{m-1} \overset{\pi_{m-1}}{\longrightarrow}
    \cdots \overset{\pi_2}{\longrightarrow} B_1 \overset{\pi_1}{\longrightarrow} B_0
   =\ \{ \text{a point} \} , \]
   where each $B_i$ ($1\leq i \leq m$) is the projectivization of
   the Whitney sum of a finite collection of real line bundles over $B_{i-1}$.
   The $B_m$ is called a \textit{generalized real Bott manifold}
      in~\cite{SuyMasDong10}. In fact, the Remark $6.5$ in~\cite{SuyMasDong10}
      told us the following. \vskip .2cm

     \begin{prop}[Choi, Masuda and Suh~\cite{SuyMasDong10}] The set of all generalized real Bott
      manifolds are exactly the set of all small covers over
      products of simplices. \vskip .2cm
    \end{prop}

    \begin{cor}
      A small cover $Q^n$ satisfies the
      condition in Theorem~\ref{thm:Main} if and only if $Q^n$
      is a generalized real Bott manifold.
   \end{cor}

     \vskip .2cm

      Moreover, we can prove the following.\vskip .2cm

    \begin{thm} \label{thm:Main-2}
    Suppose $M^n$ is a connected manifold. Then $M^n$ is a principal
     $(\Z_2)^m$-bundle over some small cover with $\mathrm{hrk}(M^n,\Z_2)
     =2^m$ if and only if $M^n$ is homeomorphic to a product of spheres.
   \end{thm} \vskip .2cm

     Obviously, if $M^n$ is homeomorphic to a product of spheres
       $S^{n_1}\times \cdots \times S^{n_k}$, then the product
      action of the antipodal map of each $S^{n_i}$ defines a free
        $(\Z_2)^k$ action on $M^n$ whose orbit space is
        $\R P^{n_1} \times \cdots\times \R P^{n_k}$. And in fact,
         $\mathrm{hrk}(M^n,\Z_2) = 2^k$ and $\R P^{n_1} \times \cdots \times \R P^{n_k}$
         is a small cover. This proves the sufficiency part of
         Theorem~\ref{thm:Main-2}.
           But the necessity part of Theorem~\ref{thm:Main-2}
          is not trivial (see Section~\ref{Sec3}).\vskip .2cm

   \begin{rem} \label{Rem:Quaternion-action}
      It is not necessarily that a closed connected manifold
      $W^n$ with a free $(\Z_2)^m$-action and
      $\mathrm{hrk}(W^n,\Z_2) =2^m$ must be
       a product of spheres. For example, Let $UT(S^{2n})$ be the unit
        tangent bundle of the $2n$-dimensional sphere $S^{2n}$ ($n\geq 1$).
        Then $UT(S^{2n})$ is a $(4n-1)$-dimensional closed manifold.
        If we think of $S^{2n}$ as the unit sphere centered at the origin in $\R^{2n+1}$, then
        the tangent space of $S^{2n}$ at any point can be thought of as a vector subspace
        of $\R^{2n+1}$. Under this viewpoint, we
         can represent any element in $UT(S^{2n})$ by $(x,v)$ where
        $x\in S^{2n}$ and $v\in T_x(S^{2n})$ is a unit tangent
        vector at $x$. We can define two free involutions $\sigma_1,\sigma_2$ on $UT(S^{2n})$ by:
        \nn
      \begin{itemize}
        \item $\sigma_1 (x,v) = (-x,-v), \ \forall\, (x,v)\in
        UT(S^{2n})$;\n

        \item $\sigma_2 (x,v) = (x,-v), \ \forall\, (x,v)\in   UT(S^{2n})$.
        \nn
      \end{itemize}

       Obviously, $\sigma_1\circ \sigma_2 = \sigma_2\circ\sigma_1$,
       so we get a free $(\Z_2)^2$-action on $UT(S^{2n})$. It is not
       hard to see that $H^i(UT(S^{2n}),\Z_2) \cong  H^i(S^{2n}\times S^{2n-1},\Z_2)$
       for all $i$. So we have $\mathrm{hrk}(UT(S^{2n}),\Z_2) = 4 =  2^2$.
        But $UT(S^{2n})$ is not homeomorphic to $S^{2n}\times
        S^{2n-1}$ since their $\Z_2$-cohomology
      ring structures are different and their rational homology groups are not isomorphic either.
         This example is informed to the author by M.~Masuda.
         By our Theorem~\ref{thm:Main-2},
      the orbit space $UT(S^{2n}) \slash (\Z_2)^2$ is not homeomorphic to any small
      cover. \vskip .2cm
   \end{rem}

   In this example, $UT(S^{2n})$ has the same
      $\Z_2$-cohomology groups as a product of spheres,
       so it is interesting to ask the following question.\vskip .2cm

 \textbf{Question:} does there exists a closed connected manifold
      $W^n$ with a free $(\Z_2)^m$-action so that (i) $\mathrm{hrk}(W^n,\Z_2) =2^m$
      and (ii) the $\Z_2$-cohomology groups of $W^n$ do not agree with
      the $\Z_2$-cohomology groups of any product of spheres?
    \\

             The paper is organized as follows.
          In section~\ref{Sec2}, we will review some basic definitions and results
          introduced in~\cite{Yu2010} and~\cite{Yu2010-2}. Then in section~\ref{Sec3},
          we will prove Theorem~\ref{thm:Main} and Theorem~\ref{thm:Main-2}.
          In particular, the~\textquotedblleft only if\,\textquotedblright~part of
          Theorem~\ref{thm:Main} uses an interesting result of
          Choi~\cite{Choi10} on the structure of simple polytopes.  \\

  \section{Some backgrounds and known results} \label{Sec2}

 Suppose $Q^n$ is an arbitrary $n$-dimensional closed connected smooth manifold. Let
  $k= \dim_{\Z_2} H_{n-1}(Q^n,\Z_2)$. It is well known that
   we can choose some embedded $(n-1)$-dimensional
  submanifolds $\Sigma_1, \cdots, \Sigma_k$ whose homology classes
  form a linear basis of $H_{n-1}(Q^n,\Z_2)$. If we cut $Q^n$ open
  along $\Sigma_1, \cdots, \Sigma_k$, i.e. we remove a small
  tubular neighborhood $N(\Sigma_i)$ of each $\Sigma_i$ and remove the interior
  of each $N(\Sigma_i)$ from $Q^n$, we will get a nice manifold
  with corners $V^n = Q^n - \bigcup_i int(N(\Sigma_i))$,
  which is called a \textit{$\Z_2$-core} of $Q^n$
  (see~\cite{Yu2010} for the details of the construction).
  the boundary of $V^n$ is the union of some compact subsets
         $P_1,\cdots, P_k$, called \textit{panels}, that satisfy the following three
         conditions:\vskip .2cm

           \begin{itemize}
             \item[(a)] each panel $P_i$ is a disjoint union of facets of $V^n$ and
               each facet is contained in exactly one panel; \vskip .1cm

             \item[(b)] there exists a free involution
                  $\tau_i$ on each $P_i$ which
                  sends a face $f \subset P_i$ to a face $f'\subset P_i$
                 (it is possible that $f'=f$); \vskip .1cm

             \item[(c)] for $\forall\, i \neq j$, $\tau_i (P_i \cap P_j) \subset P_i \cap
             P_j$ and $\tau_i\circ\tau_j = \tau_j\circ \tau_i
                 : P_i \cap P_j \rightarrow P_i \cap P_j$. \vskip .2cm
            \end{itemize}

  The $\{ \tau_i : P_i \rightarrow P_i \}_{1\leq i \leq k}$
  is called an \textit{involutive panel structure}
  on $V^n$ (see~\cite{Yu2010} for the details of the construction of $\tau_i$). \vskip .2cm

  It was shown in~\cite{Yu2010} that any principal
  $(\Z_2)^m$-bundle $M^n$ over $Q^n$ determines a map
  $\lambda: \{ P_1, \cdots, P_k \} \rightarrow (\Z_2)^m$ which is called a
  \textit{$(\Z_2)^m$-coloring} of $V^n$. And we can recover $M^n$ from $(V^n, \lambda)$
  in the following way called \textit{glue-back construction}.
       \begin{equation} \label{Glue_Back}
           M^n\cong M(V^n,\lambda) := V^n \times (\Z_2)^m \slash \sim
         \end{equation}
           Where $(x,g)\sim (x',g') $ whenever
            $x' = \tau_i(x)$ for some $P_i$ and
              $g' = g+ \lambda(P_i) \in (\Z_2)^m$.
        So if $x$ is in the relative interior of $P_{i_1} \cap \cdots \cap
             P_{i_s}$,
           $(x,g) \sim (x',g')$ if and only if
           $ (x',g')= ( \tau^{\varepsilon_s}_{i_s}\circ \cdots \circ
           \tau^{\varepsilon_1}_{i_1}(x), g + \varepsilon_1\lambda(P_1) + \cdots +
           \varepsilon_s\lambda(P_s))$
           where $\varepsilon_j= 0$ or $1$ for any $1\leq  j \leq s$ and
           $\tau^{0}_{i_j} := id$. \vskip .2cm

    Let $\theta_{\lambda} : V^n \times (\Z_2)^m \rightarrow
    M(V^n,\lambda)$ be the quotient map.
  There is a natural free $(\Z_2)^m$-action on $M(V^n,\lambda)$ defined by:
      \begin{equation} \label{Equ:FreeAction}
              g'\cdot \theta_{\lambda}(x,g) := \theta_{\lambda}(x, g'+g),\ \; \forall\, x\in V^n, \
             \forall\, g, g'\in (\Z_2)^m.
           \end{equation}
  And the homeomorphism between $M^n$ and $M(V^n,\lambda)$ is
  equivariant with respect to the free $(\Z_2)^m$-action.
   So we can represent any
 principal $(\Z_2)^m$-bundle over $Q^n$ by $M(V^n,\lambda)$ for some
 $(\Z_2)^m$-coloring $\lambda$ of $V^n$. Let
      \begin{align*}
      \quad  \mathrm{Col}_m(V^n) & := \text{the set of all $(\Z_2)^m$-colorings
       of $V^n$} \\
         L_{\lambda} &:= \text{the subgroup of $(\Z_2)^m$ generated
         by $ \{ \lambda(P_1), \cdots ,\lambda(P_k) \} $},\\
          \text{rank}(\lambda)  & :=  \mathrm{dim}_{\Z_2}
          L_{\lambda}.
       \end{align*}

        Obviously, $\mathrm{rank}(\lambda) \leq  k=\dim_{\Z_2}H_{n-1}(Q^n, \Z_2)$.
    If $\mathrm{rank}(\lambda) = k$, we call $\lambda$
     \textit{maximally independent} (in this case, we must have
     $m\geq k$).
 \vskip .2cm

  \begin{lem} [Theorem 2.3 in~\cite{Yu2010}] \label{thm:comp}
        For any $(\Z_2)^m$-coloring $\lambda$ of $V^n$,
       $M(V^n,\lambda)$ has $2^{m-\mathrm{rank}(\lambda)}$ connected
       components which are pairwise homeomorphic, and each
       connected component of $M(V^n,\lambda)$ is a principal
       $(\Z_2)^{\mathrm{rank}(\lambda)}$ bundle over $Q^n$.
  \end{lem}
   \vskip .2cm

 \begin{lem}[Lemma 2.8 in~\cite{Yu2010-2}] \label{Lem:Max_Indep}
        Suppose $\lambda_{max} \in \mathrm{Col}_k(V^n) $ is a
         maximally independent $(\Z_2)^k$-coloring on $V^n$,  where
         $k=\dim_{\Z_2}H_{n-1}(Q^n, \Z_2)$.
         Then for any $\lambda \in \mathrm{Col}_k(V^n)$,
          $ \mathrm{hrk} (M(V^n,\lambda),\Z_2) \geq
           \mathrm{hrk} (M(V^n,\lambda_{max}),\Z_2)$.
      \end{lem}
  \vskip .2cm

    Next, we review some basic facts on the small cover and real moment-angle manifold.
     Suppose $P^n$ is an $n$-dimensional simple polytope with $k+n$
     facets ($k\geq 1$). Here, simple means that each vertex of $P^n$ is incident to exactly
     $n$ facets of $P^n$. Let $F_1,\cdots, F_{k+n}$ be all the facets of
   $P^n$. For any $m\geq 1$, a map $\{ F_1,\cdots, F_{k+n} \} \rightarrow (\Z_2)^m$
   is called a $(\Z_2)^m$-\textit{coloring} of $P^n$. \vskip .2cm

     Suppose $Q^n$ is a small cover over $P^n$.
     Then $Q^n$ determines a $(\Z_2)^n$-coloring
     $\mu$ of $P^n$ that satisfies:
       whenever $F_{i_1}\cap \cdots \cap F_{i_s} \neq \varnothing$,
        $\mu(F_{i_1}), \cdots, \mu(F_{i_s})$ are
        linearly independent vectors in $(\Z_2)^{n}$.
        The $\mu$ is also called the \textit{characteristic function} of
        $Q^n$ (see~\cite{DaJan91}).  For any face
    $\mathbf{f}=F_{i_1}\cap \cdots \cap F_{i_l}$ of $P^n$, let
    $G^{\mu}_{\mathbf{f}}$ be the rank-$l$ subgroup of $(\Z_2)^n$ generated by
    $\mu(F_1),\cdots, \mu(F_l)$.
       Then we can recover $Q^n$ from $(P^n, \mu)$ by:
    \begin{equation} \label{Equ:Glue-SmallCov}
       \text{$Q^n= P^n\times (\Z_2)^n \slash
    \sim$,\ \ $(p,w)\sim (p',w') \Longleftrightarrow p=p', w-w'\in
    G^{\mu}_{\mathbf{f}(p)}$},
    \end{equation}
     where $\mathbf{f}(p)$ is the unique face of $P^n$ that contains
    $p$ in its relative interior. Let $\zeta_{\mu}: P^n\times (\Z_2)^n \rightarrow Q^n$
    be the corresponding quotient map. Then the locally standard $(\Z_2)^n$-action on
    $Q^n$ can be written as:
    \begin{equation} \label{Equ:Local-Std-Action}
      w' \cdot \zeta_{\mu} (p,w) = \zeta_{\mu}(p,w' + w),\ \, \forall\, p\in P^n,
    \ w,w'\in (\Z_2)^n
   \end{equation}

  Obviously, the orbit space of this action can be identified with $P^n$.
    It was shown in~\cite{DaJan91}
    that the $\Z_2$-Betti numbers of $Q^n$ are decided only by the
   $h$-vector of $P^n$. In particular, $H_{n-1}(Q^n,\Z_2) \cong
   (\Z_2)^k$. Moreover, any facet $F_i$ of $P^n$ is a simple polytope of dimension
   $n-1$, and $\zeta_{\mu} (F_i\times (\Z_2)^n)$
   is a small cover over $F_i$ whose characteristic function $\mu_{F_i}$ on $F_i$ is
   induced from $\mu$ by: $\mu_{F_i}(F_j\cap F_i) := \mu(F_j)$ for any face $F_j\cap F_i$
   of $F_i$. \vskip .2cm

       In addition, let $\{ e_1,\cdots , e_{k+n} \}$ be a basis of $(\Z_2)^{k+n}$ and
        define a $(\Z_2)^{k+n}$-coloring $\mu_0$ of $P^n$ by
     $\mu_0(F_i) := e_i,\ 1 \leq i \leq k+n$. Then the \textit{real moment-angle manifold}
   $\R\mathcal{Z}_{P^n}$ is obtained by gluing $2^{k+n}$ copies of $P^n$ together
   according to $\mu_0$ and the rule in~\eqref{Equ:Glue-SmallCov}.
   Let $\Theta: P^n \times (\Z_2)^{k+n} \rightarrow \R\mathcal{Z}_{P^n}$
   be the corresponding quotient map. There is a \textit{canonical}
   $(\Z_2)^{k+n}$-action on $\R\mathcal{Z}_{P^n}$ defined by:
      \begin{equation} \label{Equ:Cano-Action}
      g' \circledast \Theta (p,g) =  \Theta (p, g' + g),\ \,
      \forall \, p\in P^n,\ \forall\, g , g' \in  (\Z_2)^{k+n}.
   \end{equation}
    For the small cover $Q^n$, there exists a subtorus
    $H$ of $(\Z_2)^{k+n}$ with rank $k$ so that:
     \begin{itemize}
        \item[(i)]  $H$ acts freely on $\R\mathcal{Z}_{P^n}$ through the canonical
    action $\circledast$, and \vskip .1cm
        \item[(ii)] the orbit space
    $\R\mathcal{Z}_{P^n}\slash H$ is homeomorphic to $Q^n$.
      \end{itemize}
      \vskip .1cm

   But we remark that the subtorus $H \subset (\Z_2)^{k+n}$
   which satisfies (i) and (ii) is not unique (see~\cite{Yu2010-2}).
 \\

    \section{Proof of Theorem~\ref{thm:Main} and Theorem~\ref{thm:Main-2} } \label{Sec3}
     \vskip .2cm

  \textbf{Proof of Theorem~\ref{thm:Main}.}\
   First, suppose $Q^n$ is a small cover over a product of simplices
   $\Delta^{n_1}\times \cdots \times \Delta^{n_r}$ where $ n_1 + \cdots +
   n_r = n$. For the sake of brevity, we denote
   the simple polytope $\Delta^{n_1}\times \cdots \times \Delta^{n_r}$
   by $\Delta^I$ where $I= (n_1,\cdots, n_r)$.
    It is easy to see that the number of facets of
   $\Delta^I$ equals $r+n_1 + \cdots + n_r = r+n$,
   and $\R\mathcal{Z}_{\Delta^I} \cong S^{n_1}\times \cdots \times S^{n_r}$.
   By the discussion at the end of the previous section,
    there exists some subtorus $H \subset (\Z_2)^{r+n}$ with rank
    $r$ so that
    $H$ acts freely on $\R\mathcal{Z}_{\Delta^I}$ through the canonical
    action, and the orbit space
    $\R\mathcal{Z}_{\Delta^I}\slash H$ is homeomorphic to $Q^n$.
        In other words,
    $S^{n_1}\times \cdots \times S^{n_r}$ is a principal $(\Z_2)^r$-bundle over $Q^n$.
    Notice that $\mathrm{hrk}(S^{n_1}\times \cdots \times S^{n_r},\Z_2) = 2^r$, so we have
    proved
    the~\textquotedblleft if\,\textquotedblright~part of the Theorem~\ref{thm:Main}.
    \vskip .2cm

    Conversely, if $Q^n$ is a small cover over a simple polytope
    $P^n$.
     Let $F_1,\cdots, F_{k+n}$ be all the facets of $P^n$ and
     $\mu$ be the $(\Z_2)^n$-coloring (characteristic function) of $P^n$ corresponding to $Q^n$.
     Let $\pi_{\mu}: Q^n \rightarrow P^n$ be the
    orbit map of the locally standard $(\Z_2)^n$-action
    (see~\eqref{Equ:Local-Std-Action}).
    Now, assume that there exists a positive integer $m$ and a
     principal $(\Z_2)^m$-bundle $\xi: M^n \rightarrow Q^n$
    with $\mathrm{hrk}(M^n, \Z_2) = 2^m$. We want to show that $P^n$ must be
     a product of simplices.
     \vskip .2cm

    When $n=1$, this is obviously true.\vskip .2cm

    When $n=2$, notice that the Euler characteristics of $M^2$ and $Q^2$
    have the relation:
     $\chi(M^2) = 2^m \cdot \chi(Q^2)$. Without loss of generality,
     we can assume $M^2$ is connected (if $M^2$ is not connected, we just
     consider any one of its components). Then
     $\mathrm{hrk}(M^2,\Z_2) = 4-\chi(M^2)$.
     The assumption $\mathrm{hrk}(M^2, \Z_2) = 2^m$ implies that:
     $ 2^m (\chi(Q^2) + 1) =4$, which will force
       $\chi(Q^2) =0$ or $1$. Then $Q^2$ must be
     a torus, a Klein bottle or a real projective plane.
     The torus and Klein bottle are small covers over the square (product of
     $1$-simplices) and the real projective plane is the small cover over
     the $2$-simplex. So in any case, $P^2$ is a product of simplices.
      \vskip .2cm

   When $n \geq 3$, we claim the following.   \vskip .2cm

    \textbf{Claim:} any $2$-dimensional face of $P^n$ ($n\geq 3$) is either a triangle
    or a square. \vskip .2cm

    To prove this claim, we will use the glue-back construction
    to analyze the principal $(\Z_2)^m$-bundle $M^n$ as we did in the proof of Theorem~\ref{thm:Old-Result}
    in~\cite{Yu2010-2}.
    First, we can construct some
    special $\Z_2$-core of $Q^n$ in the following way.
    Take an arbitrary vertex $v_0$ of $P^n$ and
     assume that $F_{i_1},\cdots, F_{i_k}$ are those facets of $P^n$ which are not
     incident to $v_0$.
     Then according to~\cite{DaJan91}, $\dim_{\Z_2} H_{n-1}(Q^n,\Z_2) = k$ and
      the homology classes of the embedded submanifolds
      $\pi_{\mu}^{-1}(F_{i_1}),\cdots, \pi_{\mu}^{-1}(F_{i_k})$
      (called \textit{facial submanifolds} of $Q^n$)
       form a $\Z_2$-linear basis of $H_{n-1}(Q^n,\Z_2)$.
       Cutting $Q^n$ open along $\pi_{\mu}^{-1}(F_{i_1}),\cdots, \pi_{\mu}^{-1}(F_{i_k})$
       will give us a $\Z_2$-core $V^n$ of $Q^n$.
      \vskip .2cm

      Then our principal $(\Z_2)^m$-bundle $M^n$ over $Q^n$ is
     (equivariantly) homeomorphic to $M(V^n,\lambda)$ for some
      $(\Z_2)^m$-coloring $\lambda$ on $V^n$. So
      $\mathrm{hrk}(M(V^n,\lambda),\Z_2) = 2^m$.\vskip .2cm

     \begin{enumerate}
       \item[Case 1:] If $m \leq k $, let $\iota : (\Z_2)^m \hookrightarrow (\Z_2)^k$ be the standard
      inclusion and define $\widehat{\lambda}:= \iota \circ
      \lambda$. We consider $\widehat{\lambda}$ as a $(\Z_2)^k$-coloring on $V^n$.
        The Theorem~\ref{thm:comp} implies that $M(V^n,\widehat{\lambda})$
       consists of $2^{k-m}$ copies of $M(V^n,\lambda)$, so
       $\mathrm{hrk}(M(V^n,\widehat{\lambda}),\Z_2) = 2^k$.\vskip .2cm

      \item[Case 2:] If $m>k$, since $\mathrm{rank}(\lambda) \leq k$,
      with a proper change of basis, we can assume $L_{\lambda} \subset (\Z_2)^k \subset
       (\Z_2)^m$. Let $\varrho: (\Z_2)^m \rightarrow (\Z_2)^k$ be
       the standard projection. Define $\overline{\lambda}:= \varrho\circ \lambda$.
       Similarly, we consider $\overline{\lambda}$ as a $(\Z_2)^k$-coloring on
     $V^n$ and so by Theorem~\ref{thm:comp}, $M(V^n,\lambda)$
       consists of $2^{m-k}$ copies of $M(V^n,\overline{\lambda})$, so
       $\mathrm{hrk}(M(V^n,\lambda),\Z_2) = 2^k$.
       \end{enumerate}

       So in whatever case,
       there always exists an element $\lambda^* \in \mathrm{Col}_k(V^n)$
       so that $\mathrm{hrk}(M(V^n,\lambda^*), \Z_2) = 2^k$.
       Moreover, by Theorem~\ref{thm:Old-Result} and Lemma~\ref{Lem:Max_Indep},
      we can assume that $\lambda^*$ is maximally independent, i.e.
      $\mathrm{rank}(\lambda^*) = k$. \vskip .2cm

       Let
      $\xi_{\lambda^*} : M(V^n,\lambda^*) \rightarrow Q^n$ be the orbit map of
       the natural $(\Z_2)^k$-action on $M(V^n,\lambda^*)$ defined by~\eqref{Equ:FreeAction}.
    In the proof of Theorem~\ref{thm:Old-Result} in~\cite{Yu2010-2},
    it was shown that
    for any facet $F_i$ of $P^n$,
    $\mathrm{hrk}(M(V^n,\lambda^*), \Z_2) \geq \mathrm{hrk}(\xi^{-1}_{\lambda^*}
    (\pi^{-1}_{\mu}(F_i)),\Z_2)$.
    Notice that: \vskip .2cm

    \begin{itemize}
      \item  $\pi^{-1}_{\mu}(F_i)$ is a small cover over $F_i$;
      \vskip .1cm

      \item  $\xi^{-1}_{\lambda^*}(\pi^{-1}_{\mu}(F_i))$ is a principal
    $(\Z_2)^k$-bundle over $\pi^{-1}_{\mu}(F_i)$.
    \vskip .2cm
    \end{itemize}

    So by Theorem~\ref{thm:Old-Result}, we have
     $\mathrm{hrk}(\xi^{-1}_{\lambda^*}(\pi^{-1}_{\mu}(F_i)),\Z_2) \geq
    2^k$. But by our construction, $\mathrm{hrk}(M(V^n,\lambda^*), \Z_2) = 2^k$, so we
    must have
    $\mathrm{hrk}(\xi^{-1}_{\lambda^*}(\pi^{-1}_{\mu}(F_i)),\Z_2) =
    2^k$. Let $Y_i = \xi^{-1}_{\lambda^*}(\pi^{-1}_{\mu}(F_i))$. So $Y_i$ is a
    principal $(\Z_2)^k$-bundle over the small cover
    $\pi^{-1}_{\mu}(F_i)$ with $\mathrm{hrk}(Y_i,\Z_2) = 2^k$ (see the following diagram).\vskip .2cm

   \[ \xymatrix{
    Y_i:= \xi^{-1}_{\lambda^*}(\pi^{-1}_{\mu}(F_i)) \ar[d]^{\xi_{\lambda^*}} \ar[r]^{\quad\ \ \subset} &  M(V^n,\lambda^*) \ar[d]^{\xi_{\lambda^*}} \\
     \pi^{-1}_{\mu}(F_i)\, \ar[d]^{\pi_{\mu}} \ar[r]^{\;\ \; \subset} & \  Q^n \ar[d]^{\pi_{\mu}} \\
     F_i \ar[r]^{\subset} & P^n   }
     \]
      \vskip .3cm

    By applying the above argument to the principal $(\Z_2)^k$-bundle
      $Y_i$ over the small cover $\pi^{-1}_{\mu}(F_i)$,
       we can show that for any codimension two
     face $F_i\cap F_j$ of $P^n$, there exists some positive integer $k'$ and
     some principal $(\Z_2)^{k'}$-bundle $Y_{ij}$ over the small cover
      $\pi^{-1}_{\mu}(F_i\cap F_j)$
      with $\mathrm{hrk}(Y_{ij},\Z_2) = 2^{k'}$. \vskip .2cm

       Then by iterating this argument,
      we can show that
      for any $2$-dimensional face $\mathbf{f}$ of $P^n$,
      there exists some positive integer $k_{\mathbf{f}}$ and
      some principal $(\Z_2)^{k_{\mathbf{f}}}$-bundle $Y_{\mathbf{f}}$ over
      the small cover $\pi^{-1}_{\mu}(\mathbf{f})$
      with $\mathrm{hrk}(Y_{\mathbf{f}},\Z_2) = 2^{k_\mathbf{f}}$.
      Then our discussion on dimension two cases suggests that
       $\mathbf{f}$ must be a square or a triangle. So the claim is proved.
  \vskip .2cm

   Then our theorem follows from the Theorem~\ref{thm:Polytope}
   below, which is an unpublished result of Suyoung Choi~\cite{Choi10}.
    \hfill $\square$\\

  \begin{thm}[Choi~\cite{Choi10}] \label{thm:Polytope}
   For an $n$-dimensional simple polytope $P^n$ with $n\geq 3$,
  then $P^n$ is a product of simplices if and only if
   any $2$-dimensional face of $P^n$ is either a triangle or a
   square.
  \end{thm}

   \vskip .2cm

  \textbf{Proof of Theorem~\ref{thm:Main-2}.}
     By Theorem~\ref{thm:Main}, if $M^n$ is a principal $(\Z_2)^m$-bundle over a
     small cover $Q^n$ with $\mathrm{hrk}(M^n,\Z_2)=2^m$, $Q^n$ must
     be a small cover over a product of simplices
     $\Delta^I= \Delta^{n_1}\times \cdots \times \Delta^{n_r}$ where
     $n_1+ \cdots + n_r = n$. Let $\{ v^i_0,\cdots, v^i_{n_i} \}$ be
     the set of vertices of $\Delta^{n_i}$. Then each vertex of
     $\Delta^I$ can be written as a product of vertices of
     $\Delta^{n_i}$'s for $i=1,\cdots, r$. Hence the set of vertices
     of $\Delta^I$ is:
      \[  \{   v_{j_1\ldots j_r}= v^1_{j_1}\times \cdots \times v^r_{j_r} \,
      | \ 0\leq j_i \leq n_i,\ i=1,\cdots, r \}. \]
    Each facet of $\Delta^I$ is the product of a codimension-one
    face of $\Delta^{n_i}$'s and the remaining simplices. So the set
    of facets of $\Delta^I$ is:
      \[  \mathcal{F}(\Delta^I) = \{ F^i_{k_i} \, | \ 0 \leq k_i \leq n_i,
      \ i=1,\cdots, r  \},  \]
   where $F^i_{k_i} = \Delta^{n_1}\times \cdots \times \Delta^{n_{i-1}} \times f^{i}_{k_i}
   \times \Delta^{n_{i+1}} \times \cdots \times \Delta^{n_r}$, and
   $f^{i}_{k_i}$ is the codimension-one face of the simplex $\Delta^{n_i}$
   which is opposite to the vertex $v^i_{k_i}$. So there are
   $r+n$ facets in $\Delta^I$.  Since $\Delta^I$ is simple,
   exactly $n$ facets meet at each vertex. Indeed, the vertex
   $v_{j_1\ldots j_r}$ of $\Delta^I$ is the intersection of
   all the $n$ facets in
   $$\mathcal{F}(\Delta^I) - \{ F^i_{j_i}\, | \ i=1,\cdots, r
   \}.$$
    In particular, the $n$ facets that
    intersect at the vertex $v_{0\ldots 0}$ are:
    $$\mathcal{F}(\Delta^I) - \{ F^i_0 \, | \ i=1,\cdots, r
   \} = \{ F^1_1, \cdots, F^1_{n_1},\cdots, F^r_1, \cdots, F^r_{n_r} \} $$
   And the facets not incident to $v_{0\ldots 0}$ are
   $ F^1_0,\cdots, F^r_0 $. Note that for any $1\leq i \neq  i' \leq r$,
   the intersection of $F^i_0$ and  $F^{i'}_0 $ is exactly a codimension two
   face of $\Delta^I$.\vskip .2cm

     Suppose $\mu$ is the characteristic function of $Q^n$
     on $\Delta^I$ and $\pi_{\mu} : Q^n \rightarrow \Delta^I$
     is the corresponding quotient map the locally standard action on $Q^n$.
      Then according to the
     preceding discussion, we can cut $Q^n$ along the
     facial submanifolds $\pi^{-1}_{\mu}(F^1_0),\cdots,
     \pi^{-1}_{\mu}(F^r_0)$ which will gives us a $\Z_2$-core
     $V^n$ of $Q^n$. The panels of $V^n$ are denoted by
     $P_1,\cdots, P_r$ where $P_i$ consists of $2^n$ copies of
     $F^{i}_{0}$. \vskip .2cm

     Since we assume $M^n$ is connected, by Lemma~\ref{thm:comp},
     there exists a
     $\lambda\in \mathrm{Col}_{m}(V^n)$ such that $M^n \cong M(V^n,\lambda)$
     and $m=\mathrm{rank}(\lambda)\leq r = \dim_{\Z_2} H_{n-1}(Q^n,\Z_2)$.
     In addition,
     $\R\mathcal{Z}_{\Delta^I} = S^{n_1}\times \cdots \times S^{n_r}$
     is a principal $(\Z_2)^r$-bundle over $Q^n$.\vskip .2cm

    Let $\iota: (\Z_2)^m \rightarrow (\Z_2)^r$ be the standard inclusion and
    define $\lambda_0 = \iota \circ \lambda$. So $\lambda_0$ is
     a $(\Z_2)^r$-coloring on $V^n$. Obviously, $\mathrm{rank}(\lambda_0)=
     \mathrm{rank}(\lambda)$. Without loss of generality,
     we assume
    $\{ \lambda_0(P_1),\cdots, \lambda_0(P_m) \}$ is a basis of
    $L_{\lambda_0} \subset (\Z_2)^r$.
    Choose $\omega_1, \cdots, \omega_{r-m} \in (\Z_2)^r$
        so that $\{ \lambda_0(P_1),\cdots, \lambda_0(P_m),  \omega_1, \cdots, \omega_{r-m}  \}$
        forms a basis of $(\Z_2)^r$.
    Then we define a sequence of coloring
    $ \lambda_1,\cdots, \lambda_{r-m} \in \mathrm{Col}_r(V^n)$ as
    following: for each $1\leq j \leq r-m$,
     \[  \lambda_j(P_i) :=  \left\{
          \begin{array}{ll}
             \lambda_0(P_i),  & \hbox{$1\leq i \leq m$ or $m+j<i\leq r$;} \\
             \omega_{i-m}, & \hbox{$m+1\leq i \leq m+j$.}
          \end{array}
        \right.
       \]

   Then $\mathrm{rank}(\lambda_{j+1}) = \mathrm{rank}(\lambda_j) +1$
     for any $0 \leq j < r-m$. Let $\theta_{j} : V^n\times (\Z_2)^k \rightarrow M(V^n,\lambda_j)$
     be the quotient map of the glue-back construction.
     Then by the proof of Lemma~\ref{Lem:Max_Indep} in~\cite{Yu2010-2},
     there exists a sequence of closed
    connected manifolds:
      \begin{equation} \label{Equ:Db-Cov-Seq}
        K_{r-m} \overset{\eta_{r-m}}{\longrightarrow}
        K_{r-m-1} \overset{\eta_{r-m-1}}{\longrightarrow} \cdots
      \overset{\eta_{2}}{\longrightarrow}
        K_1 \overset{\eta_{1}}{\longrightarrow}  K_0 = M^n,
        \end{equation}
      where each $K_j = \theta_j (V^n\times L_{\lambda_j})$ is a connected component of $M(V^n,\lambda_j)$
      and the $\eta_j : K_j \rightarrow K_{j-1}$ is a double
      covering. Notice that $\mathrm{rank}(\lambda_{r-m}) = r$, so
      $M(V^n,\lambda_{r-m}) = K_{r-m}$ and $\lambda_{r-m}$ is a maximally independent
      $(\Z_2)^r$-coloring on $V^n$.
      Then both $K_{r-m}$
        and $ \R\mathcal{Z}_{\Delta^I}$ are
       connected principal $(\Z_2)^r$-bundles over $Q^n$.
       Then the Lemma 2.5 in~\cite{Yu2010-2} asserts
        that $K_{r-m}$ must be homeomorphic to $\R\mathcal{Z}_{\Delta^I}$.\vskip .2cm

        To analyze the relationship between the
      total $\Z_2$-cohomology rank of these spaces, we need the following lemma.\vskip .2cm

        \begin{lem}\label{Lem:Double_Covering}
          For a closed connected manifold $N$ and any double covering $\xi: \widetilde{N} \rightarrow N$,
         we must have $ \mathrm{hrk}(\widetilde{N},\Z_2) \leq 2 \cdot
          \mathrm{hrk}(N,\Z_2)$. And $ \mathrm{hrk}(\widetilde{N},\Z_2) =  2 \cdot
          \mathrm{hrk}(N,\Z_2)$ if and only if $\xi$ is a trivial double covering.
       \end{lem}
         \begin{proof}
          The Gysin sequence of $\xi: \widetilde{N} \rightarrow N$ in $\Z_2$-coefficient reads:
      \[     \cdots \longrightarrow  H^{i-1}(N,\Z_2) \overset{\phi_{i-1}}{\longrightarrow} H^i(N,\Z_2)
            \overset{\xi^*}{\longrightarrow} H^i(\widetilde{N},\Z_2) \longrightarrow H^i(N,\Z_2)
            \overset{\phi_i}{\longrightarrow}  \cdots
                 \]
          where $\phi_i(\gamma) = \gamma \cup e_{\xi},\ \forall\, \gamma \in H^i(N,\Z_2)$ and
         $e_{\xi}\in H^1(N,\Z_2)$ is the first Stiefel-Whitney
         class (Mod 2 Euler class) of $\xi$. Then
          by the exactness of the Gysin sequence, we have:
       \begin{align*}
            \dim_{\Z_2} H^i(\widetilde{N},\Z_2) &= \dim_{\Z_2} H^i(N,\Z_2) -
         \dim_{\Z_2}\text{Im}(\phi_{i-1}) + \dim_{\Z_2}\text{ker}(\phi_i) \\
            &= 2 \cdot \dim_{\Z_2} H^i(N,\Z_2) -  \dim_{\Z_2}\text{Im}(\phi_{i-1}) -
            \dim_{\Z_2}\text{Im}(\phi_i) \\
            & \leq  2 \cdot \dim_{\Z_2} H^i(N,\Z_2)
      \end{align*}
       So $ \mathrm{hrk}(\widetilde{N},\Z_2) \leq 2 \cdot
          \mathrm{hrk}(N,\Z_2)$. If $ \mathrm{hrk}(\widetilde{N},\Z_2)  = 2 \cdot
          \mathrm{hrk}(N,\Z_2)$, then $\text{Im}(\phi_i) =0$ for
          all $i\geq 0$. In this case, we claim the first
          Stiefel-Whitney class $e_{\xi}$  must be zero. Indeed,
          if $e_{\xi}\in H^1(N,\Z_2)$ is not zero, by the Poincar\'e
          duality for $N$, there must be some element $\alpha \in
          H^{n-1}(N,\Z_2)$ where $n$ is the dimension of $N$,
          so that $\alpha\cup e_{\xi}$ is the generator of $H^n(N,\Z_2)\cong \Z_2$.
          This would contradicts $\text{Im}(\phi_{n-1}) =0$.
          Moreover, since the first Stiefel-Whitney class completely classifies
          a double covering, so $e_{\xi}$ is zero will imply that $\xi$ is a trivial
          double covering.
       \end{proof}

       Now, let us come back to the proof of our Theorem~\ref{thm:Main-2}.
        Since
         $$\mathrm{hrk}(K_0,\Z_2) = \mathrm{hrk}(M^n,\Z_2) =
         2^m,\quad
          \mathrm{hrk}(K_{r-m},\Z_2) = \mathrm{hrk}(\R\mathcal{Z}_{\Delta^I},\Z_2) =
       2^r,$$
       so by Lemma~\ref{Lem:Double_Covering},
       we must have
       $\mathrm{hrk}(K_{j+1},\Z_2) = 2\cdot \mathrm{hrk}(K_j,\Z_2)$ for any
       $0\leq j < r-m$ in the sequence~\eqref{Equ:Db-Cov-Seq}.
       Then by Lemma~\ref{Lem:Double_Covering} again,
       each $\eta_j : K_j \rightarrow K_{j-1}$ must be a trivial double covering.
    But this is not possible since each $K_j$ is connected.
    Therefore, the only possibility is that $K_0 = K_{r-m}$, i.e.
    $r=m$ and $M^n\cong \R\mathcal{Z}_{\Delta^I}$. So $M^n$ is homeomorphic to a
    product of spheres $S^{n_1}\times \cdots \times S^{n_r}$ ($r=m$). \hfill
    $\square$
   \ \\

  \begin{rem}
   Most of the results in~\cite{Yu2010-2} and this paper have parallel
   statements for principal real torus bundles over quasitoric manifolds. The ideas
    are similar, though there are some extra
   ingredients in the later case.
  \end{rem}

   \nn

   \n

 \noindent  \textbf{Acknowledgement:}
      The author wants to thank Suyoung Choi for informing his
      result stated in Theorem~\eqref{thm:Polytope} to the author and thank
      Mikiya Masuda for some helpful discussions.

\end{document}